\documentclass[12pt]{amsart}
\usepackage{amsfonts,amssymb,amscd,amsmath,enumerate,verbatim}
\usepackage[latin1]{inputenc}
\usepackage{amscd}
\usepackage{latexsym}
\usepackage{mathptmx}
\input xy
\xyoption{all}
\def\NZQ{\mathbb}               % the font for N,Z,Q,R,C

\def\ZZ{{\NZQ Z}}
\def\RR{{\NZQ R}}

\def\frk{\mathfrak}               % font for "Fraktur"

\def\Phi{{\frk N}}
\def\ab{{\bold a}}
\def\bb{{\bold b}}
\def\cb{{\bold c}}
\def\eb{{\bold e}}
\def\tb{{\bold t}}
\def\vb{{\bold v}}
\def\ub{{\bold u}}
\def\wb{{\bold w}}
\def\xb{{\bold x}}

\def\opn#1#2{\def#1{\operatorname{#2}}} % to make operators
\opn\chara{char} 
\opn\length{\ell} 
\opn\pd{pd} 
\opn\rk{rk}
\opn\projdim{proj\,dim} 
\opn\injdim{inj\,dim} 
\opn\rank{rank}
\opn\depth{depth} 
\opn\grade{grade} 
\opn\height{height}
\opn\embdim{emb\,dim} 
\opn\codim{codim}

\opn\Tr{Tr} 
\opn\bigrank{big\,rank}
\opn\superheight{superheight}
\opn\lcm{lcm}
\opn\trdeg{tr\,deg}%\emph{
\opn\reg{reg} 
\opn\lreg{lreg} 
\opn\ini{in} 
\opn\lpd{lpd}
\opn\size{size}
\opn\mult{mult}
\opn\dist{dist}
\opn\cone{cone}
\opn\lex{lex}
\opn\rev{rev}
\opn\div{div} \opn\Div{Div} \opn\cl{cl} \opn\Cl{Cl}
\opn\Spec{Spec} \opn\Supp{Supp} \opn\supp{supp} \opn\Sing{Sing}
\opn\Ass{Ass} \opn\Min{Min}
\opn\Ann{Ann} \opn\Rad{Rad} \opn\Soc{Soc}
\opn\Syz{Syz} \opn\Im{Im} \opn\Ker{Ker} \opn\Coker{Coker}
\opn\Am{Am} \opn\Hom{Hom} \opn\Tor{Tor} \opn\Ext{Ext}
\opn\End{End} \opn\Aut{Aut} \opn\id{id} \opn\ini{in}

\opn\nat{nat}
\opn\pff{pf}%   \pf exists already
\opn\Pf{Pf} \opn\GL{GL} \opn\SL{SL} \opn\mod{mod} \opn\ord{ord}
\opn\Gin{Gin}
\opn\Hilb{Hilb}\opn\adeg{adeg}\opn\std{std}\opn\ip{infpt}
\opn\Pol{Pol}
\opn\sat{sat}
\opn\Var{Var}
\opn\Gen{Gen}

\opn\aff{aff} \opn\con{conv} \opn\relint{relint} \opn\st{st}
\opn\lk{lk} \opn\cn{cn} \opn\core{core} \opn\vol{vol}
\opn\link{link} \opn\star{star}
\opn\gr{gr}

\def\Ac{{\mathcal A}}

\def\Mc{{\mathcal M}}
\def\Hc{{\mathcal H}}

\def\Ic{{\mathcal I}}
\def\Jc{{\mathcal J}}
\def\Gc{{\mathcal G}}
\def\Fc{{\mathcal F}}
\def\Oc{{\mathcal O}}
\def\Pc{{\mathcal P}}
\def\Qc{{\mathcal Q}}

\def\Cc{{\mathcal C}}

\def\pot#1#2{#1[\kern-0.28ex[#2]\kern-0.28ex]}

\opn\dirlim{\underrightarrow{\lim}}
\opn\inivlim{\underleftarrow{\lim}}
\let\to=\rightarrow

\def\Implies{\ifmmode\Longrightarrow \else
	\unskip${}\Longrightarrow{}$\ignorespaces\fi}
\def\implies{\ifmmode\Rightarrow \else
	\unskip${}\Rightarrow{}$\ignorespaces\fi}
\def\iff{\ifmmode\Longleftrightarrow \else
	\unskip${}\Longleftrightarrow{}$\ignorespaces\fi}

\let\:=\colon
\newtheorem{Theorem}{Theorem}[section]
\newtheorem{Lemma}[Theorem]{Lemma}

\newtheorem{Proposition}[Theorem]{Proposition}

\newtheorem*{acknowledgement}{Acknowledgment}
\let\epsilon\varepsilon
\let\phi=\varphi
\let\kappa=\varkappa
\def\qed{\ifhmode\textqed\fi
	\ifmmode\ifinner\quad\qedsymbol\else\dispqed\fi\fi}
\def\textqed{\unskip\nobreak\penalty50
	\hskip2em\hbox{}\nobreak\hfil\qedsymbol
	\parfillskip=0pt \finalhyphendemerits=0}
\def\dispqed{\rlap{\qquad\qedsymbol}}

\opn\dis{dis}
\opn\height{height}
\opn\dist{dist}
\def\pnt{{\raise0.5mm\hbox{\large\bf.}}}

\opn\Lex{Lex}
\opn\conv{conv}

\begin{document}
\title{Reflexive polytopes arising from partially ordered sets and perfect graphs}
\author[T.~Hibi]{Takayuki Hibi}
\address[Takayuki Hibi]{Department of Pure and Applied Mathematics,
	Graduate School of Information Science and Technology,
	Osaka University,
	Suita, Osaka 565-0871, Japan}
\email{hibi@math.sci.osaka-u.ac.jp}
\author[A.~Tsuchiya]{Akiyoshi Tsuchiya}
\address[Akiyoshi Tsuchiya]{Department of Pure and Applied Mathematics,
	Graduate School of Information Science and Technology,
	Osaka University,
	Suita, Osaka 565-0871, Japan}
\email{a-tsuchiya@ist.osaka-u.ac.jp}
%\date{}
%\thanks{The second author was partially supported by Grant-in-Aid for JSPS Fellows 16J01549.}
\begin{abstract}
Reflexive polytopes which have the integer decomposition property are of interest. Recently, some large classes of reflexive polytopes with integer decomposition property  coming from the order polytopes and the chain polytopes of finite partially ordered sets are known.
In the present paper, we will generalize this result.
In fact, 
by virtue of the algebraic technique on Gr\"obner bases,  new classes of reflexive polytopes with the integer decomposition property coming from the order polytopes of finite partially ordered sets and the stable set polytopes of perfect graphs will be introduced. 
Furthermore, the result will give a polyhedral characterization of perfect graphs. 
Finally, we will investigate the Ehrhart $\delta$-polynomials of these reflexive polytopes.
\end{abstract}

\maketitle
\section*{Introduction}
Recently, the study on reflexive polytopes with the integer decomposition property has been achieved in the frame of Gr\"obner bases (\cite{twin,HMT1,HMT2,HTomega,HTperfect,harmony}).
First, we recall the definition of a reflexive polytope with the integer decomposition property.

%\subsection{Reflexive polytopes and the integer decomposition property}
A {\em lattice polytope} (or an {\em integral polytope}) is 
a convex polytope all of whose vertices have integer coordinates.
A lattice polytope $\Pc \subset \RR^d$ of dimension $d$ is 
called {\em reflexive} 
if the origin of $\RR^d$ belongs to the interior of $\Pc$ and 
if the dual polytope 
\[
\Pc^{\vee} = \{ {\bf x} \in \RR^{d} \, : \, \langle {\bf x}, {\bf y} \rangle \le 1,
\, \forall {\bf y} \in \Pc \}
\]
is again a lattice polytope.  Here $\langle {\bf x}, {\bf y} \rangle$
is the canonical inner product of $\RR^d$.
It is known that reflexive polytopes correspond to Gorenstein toric Fano varieties, and they are related to
mirror symmetry (see, e.g., \cite{mirror,Cox}).
The existence of a unique interior lattice point implies that in each dimension there exist only finitely many reflexive polytopes 
up to unimodular equivalence (\cite{Lag}), 
and all of them are known up to dimension $4$ (\cite{Kre}).
Moreover, every lattice polytope is unimodularly equivalent to a face of some reflexive polytope (\cite{refdim}).
On the other hand, 
we say that a lattice polytope $\Pc \subset \RR^d$ has the 
{\em integer decomposition property} if, for each integer $n \geq 1$ 
and for each $\ab \in n\Pc \cap \ZZ^d$,
where $n \Pc$ is the $n$th dilated polytope
$\{ \, n \ab \, : \, \ab \in \Pc \, \}$ of $\Pc$,  
there exist $\ab_1,\ldots,\ab_n$ belonging to $\Pc \cap \ZZ^d$ 
with $\ab=\ab_1+\cdots+\ab_n$.
The integer decomposition property is
particularly important in the theory and application of integer
programing \cite[\S 22.10]{integer}.
Hence to find new classes of reflexive polytopes with the integer decomposition property is an important problem.

%\subsection{Main results}
Given two lattice polytopes $\Pc \subset \RR^d$ and $\Qc \subset \RR^d$ 
of dimension $d$, 
we introduce the lattice polytopes $\Gamma(\Pc,\Qc) \subset \RR^{d}$ 
and $\Omega(\Pc,\Qc) \subset \RR^{d+1}$ with
\[
\Gamma(\Pc,\Qc)=\textnormal{conv}\{\Pc\cup (-\Qc)\}, 
\]\[
\Omega(\Pc,\Qc)=\textnormal{conv}\{(\Pc\times\{1\}) \cup (-\Qc\times\{-1\})\}.
\]
Here $\Pc\times\{1\} = \{ \, (\ab, 1) \in \RR^{d+1} \, : \, 
\ab \in \Pc \, \}$.
By using these constructions, we can obtain several classes of reflexive polytopes with the integer decomposition property.
In fact, in \cite{twin,HMT1,HMT2,HTomega}, large classes of reflexive polytopes with the integer decomposition property which arise from finite partially ordered sets are given.
Moreover, in \cite{HTperfect,harmony}, large classes of reflexive polytopes with the integer decomposition property which arise from perfect graphs are given.
In particular, we focus on the following result:
\begin{Theorem}[\cite{HMT1,HTomega}]
	\label{thm:oc}
	Let $P$ and $Q$ be finite partially ordered sets on $[d]=\{1,\ldots,d\}$.
	Then each of $\Gamma(\Oc_P,\Cc_Q)$ and
	$\Omega(\Oc_P,\Cc_Q)$ is a reflexive polytope with the integer decomposition property, where $\Oc_P$ is the order polytope of $P$ and $\Cc_Q$ is the chain polytope of $Q$. 
\end{Theorem}
In the present paper,
we will generalize this result.
In fact, we will show the following: 
\begin{Theorem}
	\label{main}
	Let $\Delta$ be a simplicial complex on $[d]$.  
	Then the following conditions are equivalent:
	\begin{itemize}
		\item[(i)] $\Gamma(\Oc_P,\Pc_{\Delta})$ is a reflexive polytope 
		for some finite partially ordered set $P$ on $[d]$; 
		\item[(ii)] $\Gamma(\Oc_P,\Pc_{\Delta})$ is a reflexive polytope 
		for all finite partially ordered set $P$ on $[d]$;
		\item[(iii)] $\Omega(\Oc_P,\Pc_{\Delta})$ has the integer 
		decomposition property for some finite partially ordered set $P$ 
		on $[d]$; 
		\item[(iv)] $\Omega(\Oc_P,\Pc_{\Delta})$ has the integer 
		decomposition property for all finite partially ordered set $P$ 
		on $[d]$; 
		\item[(v)] There exists a perfect graph $G$ on $[d]$ such that $\Pc_{\Delta}=\Qc_G$,
	\end{itemize}
where $\Pc_{\Delta}$ is the lattice polytope which is the convex hull of the indicator vectors of $\Delta$ and $S(G)$ is the stable set polytope of a finite simple graph $G$ on $[d]$.
	%Furthermore, if $\Delta$ is the set of the stable sets of a perfect graph, then 
	In particular,
	each of $\Gamma(\Oc_P,\Qc_{G})$ and $\Omega(\Oc_P,\Qc_{G})$ is a reflexive
	polytope with the integer decomposition property
	for all finite partially ordered set $P$ on $[d]$ and all perfect graph $G$ on $[d]$.
\end{Theorem}
%For a finite simple graph $G$ on $[d]$, $\Qc_G=\Pc_{S(G)}$ is called the \textit{stable set polytope} of $G$.
It is known that every chain polytope is a stable set polytope of a perfect graph.
Hence this theorem is a generalization of Theorem \ref{thm:oc}.
On the other hand, we can regard this theorem as a polyhedral characterization of perfect graphs. For example, this theorem implies that a finite simple graph $G$ on $[d]$ is perfect if and only if for some (or all)  finite partially ordered set  $P$ on $[d]$, $\Gamma(\Oc_P,\Qc_{G})$ is a reflexive polytope.
To find a new characterization of perfect graphs is also of interest.

\smallskip
We now turn to the discussion of Ehrhart $\delta$-polynomials of
$\Gamma(\Oc_{P},\Qc_{G})$ and $\Omega(\Oc_{P},\Qc_{G})$ for a finite partially ordered set $P$ on $[d]$ and a perfect graph $G$ on $[d]$.
Let, in general, $\Pc \subset \RR^d$ be a lattice polytope of dimension $d$.
The {\em Ehrhart $\delta$-polynomial} of $\Pc$ is the polynomial
\[
\delta(\Pc, \lambda) = 
(1 - \lambda)^{d+1} \left[ \,
1 + \sum_{n=1}^{\infty} \mid n\Pc \cap \ZZ^d \mid \, \lambda^n \, \right]
\] 
in $\lambda$.
Then the degree of $\delta(\Pc, \lambda)$ is at most $d$. Moreover, each coefficient of $\delta(\Pc, \lambda)$ is a nonnegative integer (\cite{RS_OHGCM}).
In addition $\delta(\Pc, 1)$ coincides with the \textit{normalized volume} of $\Pc$.
Refer the reader to \cite{BR15} and \cite[Part II]{HibiRedBook} for the detailed information 
about Ehrhart $\delta$-polynomials. 

Now, we recall the following result:
\begin{Theorem}[\cite{HMT2,HTomega}]
	Let $P$ and $Q$ be finite partially ordered sets on $[d]$.
	Then we have
		$$\delta(\Gamma(\Oc_P,\Cc_Q),\lambda)=\delta(\Gamma(\Cc_P,\Cc_Q),\lambda),$$
	$$\delta(\Omega(\Oc_P,\Cc_Q),\lambda)=\delta(\Omega(\Cc_P,\Cc_Q),\lambda).$$
\end{Theorem}
We will also  generalize this theorem.
In fact, we will show the following:
\begin{Theorem}
	\label{deltath}
	Let $P$ be a finite partially ordered set  on $[d]$ and $G$ a perfect graph on $[d]$.
	Then we have
	$$\delta(\Gamma(\Oc_P,\Qc_G),\lambda)=\delta(\Gamma(\Cc_P,\Qc_G),\lambda),$$
	$$\delta(\Omega(\Oc_P,\Qc_G),\lambda)=\delta(\Omega(\Cc_P,\Qc_G),\lambda),$$
	$$\delta(\Omega(\Oc_P,\Qc_G),\lambda)=(1+\lambda) \cdot \delta(\Gamma(\Oc_P,\Qc_G),\lambda).$$
\end{Theorem}

In the present paper, by using the algebraic technique on Gr\"obner bases, we will prove Theorems \ref{main} and \ref{deltath}.
Section $1$ will provide basic materials on order polytopes, chain polytopes, stable set polytopes and the toric ideals of lattice polytopes.  In addition, indispensable lemmata 
for our proof of Theorems \ref{main} and \ref{deltath} will be also reported in Section $1$. 
A proof of Theorem \ref{main} will be given in Sections $2$ and $3$.
Especially, in Section $2$, we will prove the equivalence 
(i) $\Leftrightarrow$ (ii) $\Leftrightarrow$ (v) of Theorem \ref{main} (Proposition \ref{gamma}),  and in Section $3$,
we will prove the  equivalence (iii) $\Leftrightarrow$ (iv) $\Leftrightarrow$ (v) of Theorem \ref{main}.
Note that these proofs are very similar but they are not same. 
Finally, in Section $4$, we will prove Theorem \ref{deltath}.

\begin{acknowledgement}{\rm
		The authors would like to thank anonymous referees for reading the manuscript carefully.
		The second author is partially supported by Grant-in-Aid for JSPS Fellows 16J01549.
	}
\end{acknowledgement}

\section{Preliminaries}
In this section, 
we summarize basic materials on order polytopes, chain polytopes, stable set polytopes and the toric ideals of lattice polytopes. 
First, we recall two lattice polytopes arising from a finite partially ordered set.
Let $P $ be a finite partially ordered set on  $[d]$.
A {\em poset ideal} of $P$ is a subset $I \subset [d]$
such that if $i \in I$ and $j \leq i$ in $P$, then $j \in I$.  
In particular, the empty set as well as $[d]$ is a poset ideal
of $P$.  Let $\Ic(P)$ denote the set of poset ideals of $P$.
An {\em antichain} of $P$ is a subset $A \subset [d]$
such that for all $i$ and $j$ belonging to $A$ with $i \neq j$ are incomparable in $P$.  
In particular, the empty set $\emptyset$ and each 1-element subset $\{j\}$ are antichains of $P$.
Let $\Ac(P)$ denote the set of antichains of $P$.
Given a subset $X \subset [d]$,
we may associate $\rho(X) = \sum_{j \in X} {\eb}_j \in {\RR}^d$,
where ${\bf e}_1, \ldots, \eb_{d}$ are
the standard coordinate unit vectors of ${\RR}^d$.
In particular $\rho(\emptyset)$ is the origin ${\bf 0}$ of $\RR^d$.
Stanley \cite{Stanley}
introduced two classes of lattice polytopes arising from finite partially ordered sets, 
which are called order polytopes and chain polytopes.
The \textit{order polytope} $\Oc_P$ of $P$
is defined to be the lattice polytope which is the convex hull of 
$\{\rho(I) : I \in \Jc(P) \}.$
The \textit{chain polytope} $\Cc_P$ of $P$ is defined to be the lattice polytope which is the convex hull of $\{\rho(A) : A \in \Jc(P) \}.$
It then follows that
the order polytope and the chain polytope 
are $(0,1)$-polytopes of dimension $d$.

Second, we recall a lattice polytopes arising from a finite simple graph.
Let $G$ be a finite simple graph on  $[d]$ and $E(G)$ the set of edges of $G$.
(A finite graph $G$ is called simple if $G$ has no loop and no multiple edge.)
A subset $S \subset [d]$ is called {\em stable}  
if, for all $i$ and $j$ belonging to $S$ with $i \neq j$,
one has $\{i,j\} \notin E(G)$.
Note that a stable set of $G$ is also called an {\em independent set} of $G$.
A {\em clique} of $G$ is a subset $C \subset [d]$ such that
for all $i$ and $j$ belonging to $C$ with $i \neq j$,
one has $\{i,j\} \in E(G)$.
Let us note that a clique of $G$ is 
a stable set of the complementary graph $\overline{G}$ of $G$.
The {\em chromatic number} of $G$ is the smallest integer 
$t \geq 1$ for which
there exist stable sets $S_{1}, \ldots, S_{t}$ of $G$ with
$[d] = S_{1} \cup \cdots \cup S_{t}$.
A finite simple graph $G$ is said to be {\em perfect} 
(\cite{sptheorem}) if, 
for any induced subgraph $H$ of $G$
including $G$ itself,
the chromatic number of $H$ is
equal to the maximal cardinality of cliques of $H$.
Recently, there exists two well-known combinatorial characterizations of perfect graphs.
An \textit{odd hole} is an induced odd cycle of length $\geq 5$ and an \textit{odd antihole} is the complementary graph of an odd hole.
\begin{Lemma}[{\cite[1.1]{sptheorem}}]
	\label{weak}
The complementary graph of a perfect graph is perfect.
\end{Lemma}
\begin{Lemma}[{\cite[1.2]{sptheorem}}]
	\label{strong}
	A graph is perfect if and only if it has neither odd holes nor odd antiholes as induced subgraph.
\end{Lemma}
The first characterization is called the (weak) perfect graph theorem and the second one is called the strong perfect graph theorem.
Now, we introduce the stable set polytopes of finite simple graphs. Let $S(G)$ denote the set of stable sets of $G$.
Then one has $\emptyset \in S(G)$ and $\{ i \} \in S(G)$
for each $i \in [d]$. Moreover, we can regard $S(G)$ as a simplicial complex on $[d]$.
For a simplicial complex $\Delta$ on $[d]$, we let $\Pc_{\Delta}$ be the $(0,1)$-polytope which is the convex hull of 
$\{\rho(\sigma) \, : \, \sigma \in \Delta \}$ 
in $\RR^d$. 
Then one has $\dim \Pc_{\Delta} = d$.
The {\em stable set polytope} $\Qc_G \subset \RR^{d}$
of $G$ is the lattice polytope $\Pc_{S(G)}$.
It is known that every chain polytope is the stable set polytope of a perfect graph.
In fact, for a finite partially ordered set $P$ on $[d]$, its \textit{comparability graph} $G_P$ is the finite simple graph on $[d]$ such that $\{i,j\} \in E(G_P)$ if and only if $i < j$ or $j < i$ in $P$.
Then one has $\Cc_P=\Qc_{G_P}$. 
Since every comparability graph is perfect, the class of chain polytopes
is contained in the class of the stable set polytopes of perfect graphs.

Next, we introduce the toric ideals of lattice polytopes.
Let $K[{\bf t^{\pm1}}, s] 
= K[t_{1}^{\pm1}, \ldots, t_{d}^{\pm1}, s]$
be the Laurent polynomial ring in $d+1$ variables over a field $K$. 
If $\ab = (a_{1}, \ldots, a_{d}) \in \ZZ^{d}$, then
${\bf t}^{\ab}s$ is the Laurent monomial
$t_{1}^{a_{1}} \cdots t_{d}^{a_{d}}s \in K[{\bf t^{\pm1}}, s]$. 
In particular ${\bf t}^{\bf 0}s = s$.
Let  $\Pc \subset \RR^{d}$ be a lattice polytope of dimension $d$ and $\Pc \cap \ZZ^d=\{\ab_1,\ldots,\ab_n\}$.
Then, the \textit{toric ring}
of $\Pc$ is the subalgebra $K[\Pc]$ of $K[{\bf t^{\pm1}}, s] $
generated by
$\{\tb^{\ab_1}s ,\ldots,\tb^{\ab_n}s \}$ over $K$.
We regard $K[\Pc]$ as a homogeneous algebra by setting each $\text{deg } \tb^{\ab_i}s=1$.
Let $K[\xb]=K[x_1,\ldots, x_n]$ denote the polynomial ring in $n$ variables over $K$.
The \textit{toric ideal} $I_{\Pc}$ of $\Pc$ is the kernel of the surjective homomorphism $\pi : K[\xb] \rightarrow K[\Pc]$
defined by $\pi(x_i)=\tb^{\ab_i}s$ for $1 \leq i \leq n$.
It is known that $I_{\Pc}$ is generated by homogeneous binomials.
See, e.g.,  \cite{Sturmfels}.
We refer the reader to \cite[Chapters $1$ and $5$]{dojoEN} for basic facts on toric ideals together with Gr\"obner bases.  

A lattice polytope $\Pc \subset \RR^d$ of dimension $d$ is called  \textit{compressed} if the initial ideal of the toric ideal $I_{\Pc}$ is squarefree with respect to any reverse lexicographic order (\cite{Sturmfels}).
It is known that any order polytope and any chain polytope are compressed (\cite[Theorem 1.1]{comp}).
Moreover, we know that when the stable set polytope of a finite simple graph is compressed.
\begin{Lemma}[{\cite[Theorem 2.6]{harmony}}]
	\label{comp}
	Let $\Delta$ be a simplicial complex on $[d]$.
	Then the lattice polytope $\Pc_{\Delta}$ is compressed if and only if there exists a perfect graph $G$ on $[d]$ such that $\Pc_{\Delta}=\Qc_{G}$.
\end{Lemma}

We now turn to the review of indispensable lemmata 
for our proofs of Theorems \ref{main} and \ref{deltath}.

\begin{Lemma}[{\cite[(0.1), p. 1914]{OH}}]
	\label{OHroots}
Work with the same situation above,
	Then a finite set $\Gc$ of $I_{\Pc}$ is a Gr\"{o}bner basis of $I_{\Pc}$ with respect to a monomial order $<$ on $K[\xb]$ if and only if
	$\pi(u) \neq \pi(v)$
	for all $u \notin (\{ {\rm in}_<(g) : g \in {\mathcal G} \})$ and $v\notin \langle\{ {\rm in}_<(g) : g \in {\mathcal G}\}\rangle$ 
	with $u \neq v$.
	%$\pi(u) \neq \pi(v)$ for any monomials
\end{Lemma}

\begin{Lemma}[{\cite[Lemma 1.1]{HMOS}}]
	\label{HMOS}
	Let $\Pc \subset \RR^d$ be a lattice polytope such that the origin of $\RR^d$ is contained 
	in its interior and $\Pc \cap \ZZ^d=\{\ab_1,\ldots,\ab_n \}$.
	Suppose that any lattice point in $\ZZ^{d+1}$ is a linear integer combination of the lattice points in $\Pc \times \{1\}$ and there exists an ordering of the variables $x_{i_1} < \cdots < x_{i_n}$ for which $\ab_{i_1}= \mathbf{0}$ such that the initial ideal $\textnormal{in}_{<}(I_{\Pc})$ of the toric ideal $I_{\Pc}$ with respect to the reverse lexicographic order $<$ on the polynomial ring $K[\xb]$
	induced by the ordering is squarefree.
	Then $\Pc$ is a reflexive polytope with the integer decomposition property.
\end{Lemma}

\begin{Lemma}[{\cite[Lemma 9.1.3]{Monomial}}]
	\label{flag}
	Let $\Delta$ be a simplicial complex on $[d]$.
	Then there exists a finite simple graph $G$ on $[d]$ such that $\Delta=S(G)$ is and only if $\Delta$ is flag, i.e., every minimal nonface of $\Delta$ is a 2-elements subset of $[d]$.
\end{Lemma}

\begin{Lemma}[{\cite[Corollary 35.6]{HibiRedBook}}]
	\label{facet}
	Let $\Pc \subset \RR^d$ be a lattice polytope of
	dimension $d$ containing the origin in its interior. 
	Then a point $\ab \in \RR^d$ is a vertex of $\Pc^\vee$ if 
	and only if $\Hc \cap \Pc$ is a facet of $\Pc$,
	where $\Hc$ is the hyperplane
	$$\left\{ \xb \in \RR^d : \langle \ab, \xb \rangle =1 \right\}$$
	in $\RR^d$.
\end{Lemma}

\begin{Lemma}[{\cite[Corollary 6.1.5]{Monomial}}]
	\label{lem:hilbert}
	Let $S$ be a polynomial ring and $I \subset S$ a graded ideal of $S$.
	Let $<$ be a monomial order on $S$.
	Then $S/I$ and $S/\textnormal{in}_{<}(I)$ have the same Hilbert function.
\end{Lemma}

\begin{Lemma}
	Let $\Pc \subset \RR^d$ be a lattice polytope of dimension $d$.
	If $\Pc$ has the integer decomposition property,
	then the Ehrhart $\delta$-polynomial of $\Pc$ coincides with the $h$-polynomial of $K[\Pc]$.
\end{Lemma}

\begin{Lemma}[{\cite[Theorem 1.2]{HTperfect}}]
	\label{delta}    
	Let $G_1$ and $G_2$ be perfect graphs on $[d]$.
	Then one has
	\[
	\delta(\Omega(\Qc_{G_1},\Qc_{G_2}), \lambda) 
	= (1 + \lambda)\cdot \delta(\Gamma(\Qc_{G_1},\Qc_{G_2}), \lambda).
	\]
\end{Lemma}

\section{Type $\Gamma$}
In this section, we prove the  equivalence (i) $\Leftrightarrow$ (ii) $\Leftrightarrow$ (v) of Theorem \ref{main}.
In fact, we prove the following proposition.
\begin{Proposition}
\label{gamma}
Let $\Delta$ be a simplicial complex on $[d]$.  
Then the following conditions are equivalent:
\begin{itemize}
	\item[(i)] $\Gamma(\Oc_P,\Pc_{\Delta})$ is a reflexive polytope 
	for some finite partially ordered set $P$ on $ [d]$; 
	\item[(ii)] $\Gamma(\Oc_P,\Pc_{\Delta})$ is a reflexive polytope 
	for all finite partially ordered set $P$ on $[d]$;
	\item[(iii)] There exists a perfect graph $G$ on $[d]$ such that $\Pc_{\Delta}=\Qc_G$.
\end{itemize}
In particular, 
$\Gamma(\Oc_P,\Pc_{\Delta})$ is a reflexive polytope with the integer decomposition property 
for all finite partially ordered set $P$ on $[d]$ and all perfect graph $G$ on $[d]$.
\end{Proposition}
\begin{proof}
((iii) $\Rightarrow$ (ii))
Suppose that $G$ is a perfect graph such that $\Pc_{\Delta}=\Qc_G$.
%Then we have $\Pc_{\Delta}=\Qc_{G}$. 
Let 
\[
K[\Oc\Qc] = K[\{x_{I}\}_{\emptyset \neq I \in \Jc(P)} \cup 
\{y_{S}\}_{\emptyset \neq S \in S(G)} \cup \{ z \}]
\]
denote the polynomial ring over $K$
and define the surjective ring homomorphism 
$\pi : K[\Oc\Qc] \to K[\Gamma(\Oc_P, \Qc_G)] \subset K[t_1^{\pm 1},\ldots,t_d^{\pm 1},s]$
by the following: 

\begin{itemize}
	\item
	$\pi(x_{I}) = {\bf t}^{\rho(I)}s$, 
	where $\emptyset \neq I \in \Jc(P)$;
	\item
	$\pi(y_{S}) = {\bf t}^{- \rho(S)}s$,
	where $\emptyset \neq S \in S(G)$;
	\item
	$\pi(z) = s$.
\end{itemize}
Then the toric ideal $I_{\Gamma(\Oc_P, \Qc_G)}$ of $\Gamma(\Oc_P, \Qc_G)$ is 
the kernel of $\pi$. 

Let $<_{\Oc_P}$ and $<_{\Qc_G}$ denote reverse lexicographic orders on $K[\Oc]=K[\{x_{I}\}_{\emptyset \neq I \in \Jc(P)} \cup \{ z \}]$ and $K[\Qc]=K[ 
\{y_{S}\}_{\emptyset \neq S \in S(G)} \cup \{ z \}]$
induced by
\begin{itemize}
	\item
	$z<_{\Oc_P} x_{I}$ and $z<_{\Qc_G} y_{S}$;
	\item
	$x_{I'} <_{\Oc_P} x_{I}$ if $I' \subset I$;
	\item
	$y_{S'} <_{\Qc_G} y_{S}$ if $S' \subset S$,
\end{itemize}
where $I,I' \in \Jc(P) \setminus \{\emptyset\}$ with $I \neq I'$ and $S,S' \in S(G) \setminus \{\emptyset\}$ with $S \neq S'$.
Since $\Oc_P$ and $\Qc_G$ are compressed, we know that $\text{in}_{<_{\Oc_P}}(I_{\Oc_P})$ and $\text{in}_{<_{\Qc_G}}(I_{\Qc_G})$ are squarefree.
Let $\Mc_{\Oc_P}$ and $\Mc_{\Qc_G}$ be the minimal sets of squarefree monomial generators of 
$\text{in}_{<_{\Oc_P}}(I_{\Oc_P})$ and $\text{in}_{<_{\Qc_G}}(I_{\Qc_G})$.
Then from \cite{Hibi1987}, it follows that 
\begin{equation}
\Mc_{\Oc_P}=\{x_Ix_{I'} : I,I' \in \Jc(P), I \nsubseteq I', I \nsupseteq I'\}.
\end{equation}
Let $<$ be a reverse lexicographic order on $K[\Oc\Qc]$ induced by
\begin{itemize}
	\item
	$z<y_{S} <x_{I}$;
	\item
	$x_{I'} < x_{I}$ if $I' \subset I$;
	\item
	$y_{S'} < y_{S}$ if $S' \subset S$,
\end{itemize}
where $I,I' \in \Jc(P) \setminus \{\emptyset\}$ with $I \neq I'$ and $S,S' \in S(G) \setminus \{\emptyset\}$ with $S \neq S'$,
and set
	$$\Mc=\Mc_{\Oc_P} \cup \Mc_{\Qc_G} \cup \{x_Iy_S : I \in \Jc(P),S \in S(G), \text{max}(I) \cap S \neq \emptyset\}.$$
	%Let $K[[-B, A]^{*}] \subset K[ t_{1}^{\pm 1}, \ldots, t_{d+1}^{\pm 1}, s]$
	%be the toric ring of $[-B, A]^{*}$
Let $\Gc$ be a finite set of binomials belonging to $I_{\Gamma(\Oc_P,\Qc_G)}$ with ${\mathcal M} = \{ {\rm in}_{<}(g) : g \in {\mathcal G}\}$.
	
	Now, we prove that $\Gc$ is a Gr\"obner base of $ I_{\Gamma(\Oc_P,\Qc_G)}$
	with respect to $<$.
%	Suppose that % the initial ideal 
%	${\rm in}_{<_{\rm rev}}
%	( I_{[-B, A]^{*}})$ cannot be generated by the set
%	of squarefree monomials
Suppose that there exists a nonzero irreducible binomial $f = u - v$  belonging to $I_{\Gamma(\Oc_P,\Qc_G)}$ such that 
	%$\pi(u) \neq \pi(v)$ for any monomials
	$u \notin ( \{{\rm in}_<(g) : g \in {\mathcal G}\})$ and $v\notin( \{{\rm in}_<(g) : g \in {\mathcal G}\})$ with $u \neq v$.
	Write
	\[
	u = \left(\,
	\prod_{1 \leq i \leq a} x_{I_i}^{\mu_{I_i}} 
	\right)
	\left(\, 
	\prod_{1 \leq j \leq b} y_{S_j}^{\nu_{S_j}} 
	\right),
	\, \, \, \, \, \, 
	v =
	z^{\alpha} 
	\left(\,
	\prod_{1 \leq i \leq a'} x_{I'_{i}}^{\mu'_{I'_i}} 
	\right)
	\left(\, 
	\prod_{1 \leq j \leq b'} y_{S'_j}^{\nu'_{S'_j}} 
	\right),
	\]
	where 
	\begin{itemize}
		\item $I_1,\ldots,I_a,I'_1,\ldots,I'_{a'} \in \Jc(P) \setminus \{\emptyset\}$;
		\item $S_1,\ldots,S_b,S'_1,\ldots,S'_{b'} \in S(G) \setminus \{\emptyset\}$;
		\item $a,a',b,b'$ and $\alpha$ are nonnegative integers;
		\item $\mu_{I}, \mu'_{I'}, \nu_S,\nu'_{S'}$ are positive integers.
 	\end{itemize} 
	By (1),  we may assume that
	$I_1 \subsetneq \cdots \subsetneq I_a$ and $I'_1 \subsetneq \cdots \subsetneq I'_{a'}$.
	If $(a,a')=(0,0)$, then $\text{in}_{<_{\Qc_{G}}}(f)=\text{in}_{<}(f)$.
	Hence we have $(a,a') \neq (0,0)$. 
	Assume that $I_{a} \setminus I_{a'} \neq \emptyset$.
	Then there exists a maximal element $i$ of $I_{a}$ such that $i \notin I_{a'}$.
	Hence we have
	$$\sum_{\stackrel{I \in \{I_1,\ldots,I_a\}}{i \in \max(I)}}\mu_I-\sum_{\stackrel{S \in \{S_1,\ldots,S_b\}}{i \in S}}\nu_S=-\sum_{\stackrel{S' \in \{S'_1,\ldots,S'_{b'}\}}{i \in S'}}\nu'_{S'} \leq 0.$$
	This implies that there exists a stable set $S \in \{S_1,\ldots,S_b\}$ such that $i \in S$.
	Then $x_{I_a}y_{S} \in \Mc$, a contradiction.
	Similarly, it does not follow that $I_{a'} \setminus I_{a} \neq \emptyset$.
	Therefore, by using Lemma \ref{OHroots}, $\Gc$ is a Gr\"obner base of $I_{\Gamma(\Oc_P,\Qc_G)}$
	with respect to $<$.
	Thus, by Lemma \ref{HMOS}, it follows that
	$\Gamma(\Oc_P,\Qc_G)$ is a reflexive polytope with the integer decomposition property. 
	
	((i) $\Rightarrow$ (iii))
First, suppose that there is no finite simple graph $G$ on $[d]$ with $\Delta=S(G)$.
	Then from Lemma \ref{flag}, it follows that there exists a minimal nonface $L$ of $\Delta$ with $|L| \geq 3$. 
	By renumbering the vertex set of $\Delta$, we can assume that $L=[\ell]$ with $\ell \geq 3$.
	Then the hyperplane $\Hc' \subset \RR^d$ defined by the equation $z_1+\cdots+z_{\ell}=-\ell+1$ is a supporting hyperplane of $\Gamma(\Oc_P,\Pc_{\Delta})$ since for all $i \in L$, we have $L \setminus\{i\} \in \Delta$ and $L \notin \Delta$.
	Let $\Fc$ be a facet of $\Gamma(\Oc_P,\Pc_{\Delta})$ with $\Hc' \cap \Gamma(\Oc_P,\Pc_{\Delta}) \subset \Fc$
	and $a_1z_1+\cdots+a_dz_d=1$ with each $a_i \in \RR$ the equation of the supporting hyperplane $\Hc \subset \RR^d$ with $\Fc \subset \Hc$.
	Since $-\rho(L \setminus \{i\}) \in \Hc$ for all $i \in L$, we obtain
	$-\sum_{j \in L \setminus \{i\}} a_j=1$.
		Hence $-(\ell-1)(a_1+\cdots+a_{\ell})=\ell$.
		Thus $a_1+\cdots+a_{\ell} \notin \ZZ$.
		This implies that there exists $1 \leq i \leq \ell$ such that $a_i \notin \ZZ$.
		Therefore, by Lemma \ref{facet}, it follows that $\Gamma(\Oc_P,\Pc_{\Delta})$ is not reflexive.
	
	Next, we suppose that $G$ is a non-perfect finite simple graph on $[d]$ with $\Delta=S(G)$, i.e., $\Pc_{\Delta}=\Qc_G$.
	By Lemma \ref{strong},  $G$ has either an odd hole or an odd antihole.
Suppose that  $G$ has an odd hole $C$ of length $2\ell+1$, where $\ell \geq 2$.
By renumbering the vertex set of $G$, we may assume that  the edge set of $C$ is $\{\{i,i+1\} : 1\leq i \leq 2\ell \} \cup \{1,2\ell+1\} $.
	Then the hyperplane $\Hc' \subset \RR^d$ defined by the equation $z_1+\cdots+z_{2\ell+1}= -\ell$ is a supporting hyperplane of 
	$\Gamma(\Oc_P,\Qc_G)$.
	Let $\Fc$ be a facet of $\Gamma(\Oc_P,\Qc_G)$ with $\Hc' \cap \Gamma(\Oc_P,\Qc_G) \subset \Fc$
	and $a_1z_1+\cdots+a_dz_d=1$ with each $a_i \in \RR$ the equation of the supporting hyperplane $\Hc \subset \RR^d$ with $\Fc \subset \Hc$.
	The maximal stable sets of $C$ are
	$$S_1=\{1,3,\ldots,2\ell-1\},S_2=\{2,4,\ldots,2\ell \},\ldots,S_{2\ell+1}=\{2\ell+1,2,4,\ldots,2\ell-2\}$$
	and each $i \in [2\ell+1]$ appears $\ell$ times in the above list.
	Since for each $S_i$, we have $-\sum_{j \in S_i}a_j=1$, it follows that $-\ell(a_1+\cdots+a_{2\ell+1})=2\ell +1$.
	Hence $a_1+\cdots+a_{2\ell+1} \notin \ZZ$.
	%This implies there exists $1 \leq i \leq 2\ell+1$ such that $a_i \notin \ZZ$.
	Therefore, $\Gamma(\Oc_P,\Qc_G)$ is not reflexive.
	
	Finally, we suppose that $G$ has an odd antihole $C$ such that the length of $\overline{C}$ equals $2\ell+1$, where $\ell \geq 2$.
	Similarly, we may assume that the edge set of $\overline{C}$ is $\{\{i,i+1\} : 1\leq i \leq 2\ell \} \cup \{1,2\ell+1\}$.
	Then the hyperplane $\Hc' \subset \RR^d$ defined by the equation $z_1+\cdots+z_{2\ell+1}= -2$ is a supporting hyperplane of 
	$\Gamma(\Oc_P,\Qc_G)$.
	Let $\Fc$ be a facet of $\Gamma(\Oc_P,\Qc_G)$ with $\Hc' \cap \Gamma(\Oc_P,\Qc_G) \subset \Fc$
	and $a_1z_1+\cdots+a_dz_d=1$ with each $a_i \in \RR$ the equation of the supporting hyperplane $\Hc \subset \RR^d$ with $\Fc \subset \Hc$.
	Then since the maximal stable sets of $C$ are the edges of $\overline{C}$,
   for each edge $\{i,j\}$ of $C$, we have $-(a_i+a_j)=1$.
   Hence  it follows that $-2(a_1+\cdots+a_{2\ell+1})=2\ell +1$.
	Thus $a_1+\cdots+a_{2\ell+1} \notin \ZZ$.
	%This implies there exists $1 \leq i \leq 2\ell+1$ such that $a_i \notin \ZZ$.
	Therefore, $\Gamma(\Oc_P,\Qc_G)$ is not reflexive,
	as desired.
\end{proof}

\section{Type $\Omega$}
In this section, we prove the  equivalence (iii) $\Leftrightarrow$ (iv) $\Leftrightarrow$ (v) of Theorem \ref{main}.
In fact, we prove the following proposition.
\begin{Proposition}
	\label{omega}
Let $\Delta$ be a simplicial complex on the vertex set $[d]$.  
Then the following conditions are equivalent:
\begin{itemize}
	\item[(i)] $\Omega(\Oc_P,\Pc_{\Delta})$ has the integer 
	decomposition property for some finite partially ordered set $P$ 
	on $[d]$; 
	\item[(ii)] $\Omega(\Oc_P,\Pc_{\Delta})$ has the integer 
	decomposition property for all finite partially ordered set $P$ 
	on $[d]$; 
	\item[(iii)] There exists a perfect graph $G$ on $[d]$ such that $\Pc_{\Delta}=\Qc_{G}$.
\end{itemize}
In particular, 
if $\Omega(\Oc_P,\Qc_{G})$ is a reflexive polytope with the integer decomposition property for all finite partially ordered set $P$ on $[d]$ and for all perfect graph $G$ on $[d]$.

\end{Proposition}
\begin{proof}
		((iii) $\Rightarrow$ (ii))
		Suppose that $G$ is a perfect graph on $[d]$ such that $\Pc_{\Delta}=\Qc_{G}$.
%		Then we have $\Pc_{\Delta}=\Qc_{G}$. 
	Let 
	\[
	K[\Oc\Qc] = K[\{x_{I}\}_{ I \in \Jc(P)} \cup 
	\{y_{C}\}_{ C \in S(G)} \cup \{ z \}]
	\]
	denote the polynomial ring over $K$
	and define the surjective ring homomorphism 
	$\pi : K[\Oc\Qc] \to K[\Omega(\Oc_P, \Qc_G)] \subset K[t_1^{\pm 1},\ldots,t_{d+1}^{\pm 1},s]$
	by the following: 
	
	\begin{itemize}
		\item
		$\pi(x_{I}) = {\bf t}^{\rho(I)}t_{d+1}s$, 
		where $I \in \Jc(P)$;
		\item
		$\pi(y_{S}) = {\bf t}^{- \rho(S)}t_{d+1}^{-1}s$,
		where $S \in S(G)$;
		\item
		$\pi(z) = s$.
	\end{itemize}
	Then the toric ideal $I_{\Omega(\Oc_P, \Qc_G)}$ of $\Omega(\Oc_P, \Qc_G)$ is 
	the kernel of $\pi$. 
	
	Let $<_{\Oc_P}$ and $<_{\Qc_G}$ denote reverse lexicographic orders on $K[\Oc]=K[\{x_{I}\}_{I \in \Jc(P)}]$ and $K[\Qc]=K[ 
	\{y_{S}\}_{ S \in S(G)} ]$
	induced by
	\begin{itemize}
		\item
		$x_{I'} <_{\Oc_P} x_{I}$ if $I' \subset I$;
		\item
		$y_{S'} <_{\Qc_G} y_{S}$ if $S' \subset S$,
	\end{itemize}
where $I,I' \in \Jc(P)$ with $I \neq I'$ and $S, S' \in S(G)$ with $S \neq S'$.
Since $\Oc_P$ and $\Qc_G$ are compressed, we know that 
	$\text{in}_{<_{\Oc_P}}(I_{\Oc_P})$ and $\text{in}_{<_{\Qc_G}}(I_{\Qc_G})$ are squarefree.
	Let	$\Mc_{\Oc_P}$ and $\Mc_{\Qc_G}$ be the minimal sets of squarefree monomial generators of 
	$\text{in}_{<_{\Oc_P}}(I_{\Oc_P})$ and $\text{in}_{<_{\Qc_G}}(I_{\Qc_G})$.
	 Then it follows that 
	\begin{equation}
		\Mc_{\Oc_P}=\{x_Ix_{I'} : I,I' \in \Jc(P), I \nsubseteq I', I \nsupseteq I'\}.
	\end{equation}
	Let $<$ be a reverse lexicographic order on $K[\Oc\Qc]$ induced by
	\begin{itemize}
		\item
		$z<y_{S} <x_{I}$;
		\item
		$x_{I'} < x_{I}$ if $I' \subset I$;
		\item
		$y_{S'} < y_{S}$ if $S' \subset S$,
	\end{itemize}
where $I,I' \in \Jc(P)$ with $I \neq I'$ and $S, S' \in S(G)$ with $S \neq S'$,
	and set
	$$\Mc=\Mc_{\Oc_P} \cup \Mc_{\Qc_G} \cup \{x_Iy_S : I \in \Jc(P),S \in S(G), \text{max}(I) \cap S \neq \emptyset\} \cup \{x_{\emptyset}y_{\emptyset}\}.$$
	%Let $K[[-B, A]^{*}] \subset K[ t_{1}^{\pm 1}, \ldots, t_{d+1}^{\pm 1}, s]$
	%be the toric ring of $[-B, A]^{*}$
	Let $\Gc$ be a finite set of binomials belonging to $I_{\Omega(\Oc_P,\Qc_G)}$ with ${\mathcal M} = \{ {\rm in}_{<}(g) : g \in {\mathcal G}\}$.
	
	Now, we prove that $\Gc$ is a Gr\"obner base of $I_{\Omega(\Oc_P,\Qc_G)}$
	with respect to $<$.
	%	Suppose that % the initial ideal 
	%	${\rm in}_{<_{\rm rev}}
	%	( I_{[-B, A]^{*}})$ cannot be generated by the set
	%	of squarefree monomials
	%Since $\Gc$ with ${\mathcal M} = \{ {\rm in}_<(f) : f \in {\mathcal G}\}$
	%is not a Gr\"obner basis, it follows that
	%there exists a nonzero irreducible binomial 
	Suppose that there exists a nonzero irreducible binomial $f = u - v$  belonging to $I_{\Omega(\Oc_P,\Qc_G)}$ such that 
	%$\pi(u) \neq \pi(v)$ for any monomials
	$u \notin ( \{{\rm in}_<(g) : g \in {\mathcal G}\})$ and $v\notin( \{{\rm in}_<(g) : g \in {\mathcal G}\})$ with $u \neq v$.
	Write
	\[
	u = \left(\,
	\prod_{1 \leq i \leq a} x_{I_i}^{\mu_{I_i}} 
	\right)
	\left(\, 
	\prod_{1 \leq j \leq b} y_{S_j}^{\nu_{S_j}} 
	\right),
	\, \, \, \, \, \, 
	v =
	z^{\alpha} 
	\left(\,
	\prod_{1 \leq i \leq a'} x_{I'_{i}}^{\mu'_{I'_i}} 
	\right)
	\left(\, 
	\prod_{1 \leq j \leq b'} y_{S'_j}^{\nu'_{S'_j}} 
	\right),
	\]
	where 
	\begin{itemize}
		\item $I_1,\ldots,I_a,I'_1,\ldots,I'_{a'} \in \Jc(P)$;
		\item $S_1,\ldots,S_b,S'_1,\ldots,S'_{b'} \in S(G)$;
		\item $a,a',b,b'$ and $\alpha$ are nonnegative integers with $(a,a') \neq 0$;
		\item $\mu_{I}, \mu'_{I'}, \nu_S,\nu'_{S'}$ are positive integers.
	\end{itemize}  
	By (2),  we may assume that
	$I_1 \subsetneq \cdots \subsetneq I_a$ and $I'_1 \subsetneq \cdots \subsetneq I'_{a'}$
	%Suppose that $\pi(u)=\pi(v)$.
	%and $(a,a') \neq (0,0)$.
	
	%Suppose that $I_{a} \setminus I_{a'} \neq \emptyset$.
	%Then there exists a maximal element $i$ of $I_{a}$ such that $i \notin I_{a'}$.
	%Hence we have
	%$$\sum_{\stackrel{I \in \{I_1,\ldots,I_a\}}{i \in \max(I)}}\mu_I-\sum_{\stackrel{C \in \{C_1,\ldots,C_b\}}{i \in C}}\nu_J=-\sum_{\stackrel{C' \in \{C'_1,\ldots,C'_{b'}\}}{i \in C'}}\nu'_{C'} \leq 0.$$
	%This implies that there exists a stable set $C \in \{C_1,\ldots,C_b\}$ such that $i \in C$.
	%Then $x_{I_a}y_{C} \in \Mc$, a contradiction.
	%Similarly, it does not follow that $I_{a'} \setminus I_{a} \neq \emptyset$.
	%
	By the same way of the proof of Proposition \ref{gamma}, 
	we know that $a=0$ and $I_{a'}=\emptyset$, or $a'=0$ and $I_{a}=\emptyset$.
    Suppose that $a=0$ and $I_{a'}=\emptyset$.
    Then by focusing on the degree of $t^{d+1}$ and $s$ of $\pi(u)$ and $\pi(v)$,
    we have
    	$$-\sum_{1 \leq j \leq b}\nu_{S_j}=\mu'_{\emptyset}-\sum_{1\leq j \leq b'}\nu'_{S'_{j}},$$
    	$$\sum_{1 \leq j \leq b}\nu_{S_j}=\alpha+\mu'_{\emptyset}+\sum_{1 \leq j \leq b'}\nu'_{S'_j}.$$
    	Hence $0=\alpha+2\mu'_{\emptyset}>0$, a contradiction.
    	
    	Suppose that $a'=0$ and $I_{a}=\emptyset$.
    	Then we have 
    		$$\mu_{\emptyset}-\sum_{1 \leq j \leq b}\nu_{S_j}=-\sum_{1 \leq j \leq b'}\nu'_{S'_j},$$
    		$$\mu_{\emptyset}+\sum_{1 \leq j \leq b}\nu_{S_j}=\alpha+\sum_{1 \leq j \leq b'}\nu'_{S'_j}.$$
    		Hence one obtains $2\mu_{\emptyset}=\alpha$.
    		By focusing on $y_{\emptyset}^{\mu_{\emptyset}} \cdot f$,
    		it is easy to show that $$f'=\left(\, 
    		\prod_{1 \leq j \leq b} y_{S_j}^{\nu_{S_j}} 
    		\right)-y_{\emptyset}^{\mu_{\emptyset}}\left(\, 
    		\prod_{1 \leq j \leq b'} y_{S'_j}^{\nu'_{S'_j}} 
    		\right) \in I_{\Omega(\Oc_P,\Qc_G)}.$$
    		Since $x_{\emptyset}y_{\emptyset} \in \Mc$, for each $i$, $S_{i} \neq \emptyset$.
    		Hence $\text{in}_{<}(f')= 
    		\prod_{1 \leq j \leq b} y_{S_j}^{\nu_{S_j}}$ and $\text{in}_{<}(f')$ divides $u$, a contradiction.
    		Therefore, by Lemma \ref{OHroots},  $\Gc$ is a Gr\"obner base of $ I_{\Omega(\Oc_P,\Qc_G)}$
    		with respect to $<$.
    		Thus, by Lemma \ref{HMOS}, it follows that
    		$\Omega(\Oc_P,\Qc_G)$ is a reflexive polytope with the integer decomposition property. 
    		
    		((i) $\Rightarrow$ (iii))
    			First, suppose that $\Delta$ is not flag.
    			Then we can assume that $L=[\ell]$ with $\ell \geq 3$ is a minimal nonface of $\Delta$.
    			For $1 \leq i \leq \ell$, we set $\ub_i=-\sum_{j \in L \setminus \{i\}}\eb_j-\eb_{d+1}$.
    			Since $L \setminus \{i\} \in \Delta$ for all $i \in L$, 
    			each $\ub_i$ is a vertex of $\Omega(\Oc_P,\Pc_{\Delta})$.
    			Then one has
    			$$\ab=\dfrac{\ub_1+\cdots+\ub_{\ell}+\eb_{d+1}}{\ell-1}=-(\eb_1+\cdots+\eb_{\ell}+\eb_{d+1}).$$
    			Since $1 < (\ell+1)/(\ell-1) \leq 2$,  one has $\ab \in 2\Omega(\Oc_P, \Pc_{\Delta})$.
    			Now, suppose that there exist $\ab_1,\ab_2 \in \Omega(\Oc_P, \Pc_{\Delta}) \cap \ZZ^{d+1}$
    			such that $\ab=\ab_1+\ab_2$.
    			Then we have $\ab \in (-\Pc_{\Delta}\times \{-1\})$.
    			However, since $L \notin \Delta$, this is a contradiction.
    			Hence $\Omega(\Oc_P, \Pc_{\Delta})$ does not have the integer decomposition property.
    			
    		Next, suppose that $G$ is a non-perfect finite simple graph on $[d]$ with $\Delta=S(G)$.
    		Now, we consider the case where $G$ has an odd hole $C$ of length $2\ell+1$, where $\ell \geq 2$. We may assume that the edge set of $C$ is $\{\{i,i+1\} : 1\leq i \leq 2\ell \} \cup \{1,2\ell+1\} $.
    		%Then the maximal stable sets of $C$ are
    		%$$S_1=\{1,3,\ldots,2\ell-1\},S_2=\{2,4,\ldots,2\ell \},\ldots,S_{2\ell+1}=\{2\ell+1,2,4,\ldots,2\ell-2\}$$
    		%and each $i \in [2\ell+1]$ appears $\ell$ times in the above list.
    		Let $S_1,\ldots,S_{2\ell+1}$ be the maximal stable sets of $C$ and 
    		for $1\leq i \leq 2\ell+1$,  
    		$\vb_{i}=-(\sum_{j \in S_i}\eb_j+\eb_{d+1})$.
    		Then one has
    		$$\bb=\dfrac{\vb_1+\cdots+\vb_{2\ell+1}+\eb_{d+1}}{\ell}=-(\eb_1+\cdots+\eb_{2\ell+1}+2\eb_{d+1}).$$
    		Since $2 <(2\ell+2)/\ell \leq 3$, $\bb \in 3\Omega(\Oc_P, \Qc_{G})$.
    		Now, suppose that there exist $\bb_1,\bb_2,\bb_3 \in \Omega(\Oc_{P}, \Qc_{G}) \cap \ZZ^{d+1}$ such that $\bb=\bb_1+\bb_2+\bb_3$.
    		Then we may assume that $\bb_1,\bb_2 \in (-\Qc_{C}\times \{-1\})$ and $\bb_3={\bf 0}$. 
    		However, since the maximal cardinality of the stable sets of $C$ equals $\ell$, this is a contradiction. Hence $\Omega(\Oc_P,\Qc_G)$ does not have the integer decomposition property.
    		
    		Finally, suppose $G$ has an odd antihole $C$ such that the length of  $\overline{C}$ equals $2\ell+1$, where $\ell \geq 2$.
    		Similarly, we may assume that the edge set of $\overline{C}$ is $\{\{i,i+1\} : 1\leq i \leq 2\ell \} \cup \{1,2\ell+1\} $.
    		%Then the maximal stable sets of $C$ is the edges of $\overline{C}$.
    		For $1 \leq i \leq 2\ell$, we set $\wb_i=-(\eb_i+\eb_{i+1}+\eb_{d+1})$ and set $\wb_{2\ell+1}=-(\eb_1+\eb_{2\ell+1}+\eb_{d+1})$.
    		Then one has
    		$$\cb=\dfrac{\wb_1+\cdots+\wb_{2\ell+1}+\eb_{d+1}}{2}=-(\eb_1+\cdots+\eb_{2\ell+1}+\ell\eb_{d+1})$$
    		and $\cb \in (\ell+1)\Omega(\Oc_{P}, \Qc_{G})$.
    		Now suppose that there exist $\cb_1,\ldots,\cb_{\ell+1} \in \Omega(\Oc_P,\Qc_G) \cap \ZZ^{d+1}$ such that $\cb=\cb_1+\cdots+\cb_{\ell+1}$.
    		Then we may assume that $\cb_1,\ldots,\cb_{\ell} \in (-\Qc_{G} \times \{-1\})$ and $\cb_{\ell+1}={\bf 0}$. 
    		However, since the maximal cardinality of the stable sets of $C$  equals $2$, this is a contradiction.
    		Hence $\Omega(\Oc_P,\Qc_G)$ does not have the integer decomposition property, as desired.
\end{proof}

\section{Ehrhart $\delta$-polynomial}
In this section, we show Theorem \ref{deltath}.
\begin{proof}[Proof of Theorem \ref{deltath}]
	Let 
	\[
K[\Cc\Qc] = K[\{x_{\max(I)}\}_{\emptyset \neq I \in \Jc(P)} \cup 
	\{y_{S}\}_{\emptyset \neq S \in S(G)} \cup \{ z \}]
	\]
	denote the polynomial ring over $K$
	and define the surjective ring homomorphism 
	$\pi : K[\Cc\Qc] \to K[\Gamma(\Cc_P, \Qc_G)] \subset K[t_1^{\pm 1},\ldots,t_d^{\pm 1},s]$
	by the following: 
	\begin{itemize}
		\item
		$\pi(x_{\max(I)}) = {\bf t}^{\rho(\max(I))}s$, 
		where $\emptyset \neq I \in \Jc(P)$;
		\item
		$\pi(y_{S}) = {\bf t}^{- \rho(S)}s$,
		where $\emptyset \neq S \in S(G)$;
		\item
		$\pi(z) = s$.
	\end{itemize}
	Then the toric ideal $I_{\Gamma(\Cc_P, \Qc_G)}$ of $\Gamma(\Cc_P, \Qc_G)$ is 
	the kernel of $\pi$. 
	
	Let $<_{\Cc_P}$ and $<_{\Qc_G}$ denote reverse lexicographic orders on $K[\Cc]=K[\{x_{\max(I)}\}_{\emptyset \neq I \in \Jc(P)} \cup \{ z \}]$ and $K[\Qc]=K[ 
	\{y_{S}\}_{\emptyset \neq S \in S(G)} \cup \{ z \}]$
	induced by
	\begin{itemize}
		\item
		$z<_{\Cc_P} x_{\max(I)}$ and $z<_{\Qc_G} y_{S}$;
		\item
		$x_{\max(I')} <_{\Cc_P} x_{\max(I)}$ if $I' \subset I$;
		\item
		$y_{S'} <_{\Qc_G} y_{S}$ if $S' \subset S$,
	\end{itemize}
where $I,I' \in \Jc(P) \setminus \{\emptyset \}$ with $I \neq I'$ and $S,S' \in S(G) \setminus \{\emptyset\}$ with $S \neq S'$.
Since $\Cc_P$ and $\Qc_G$ are compressed,
we know that
$\text{in}_{<_{\Cc_P}}(I_{\Cc_P})$ and $\text{in}_{<_{\Qc_G}}(I_{\Qc_G})$ are squarefree.
Let $\Mc_{\Cc_P}$ and $\Mc_{\Qc_G}$ be the minimal sets of squarefree monomial generators of 
$\text{in}_{<_{\Cc_P}}(I_{\Cc_P})$ and $\text{in}_{<_{\Qc_G}}(I_{\Qc_G})$.
Then from \cite{chain}, it follows that 
	\begin{equation}
	\Mc_{\Cc_P}=\{x_{max(I)}x_{\max(I')} : I,I' \in \Jc(P), I \nsubseteq I', I \nsupseteq I'\}.
	\end{equation}
	Let $<$ be a reverse lexicographic order on $K[{\bf x}, {\bf y} ,z]$ induced by
	\begin{itemize}
		\item
		$z<y_{S} <x_{\max(I)}$;
		\item
		$x_{\max(I')} < x_{\max(I)}$ if $I' \subset I$;
		\item
		$y_{S'} <  y_{S}$ if $S' \subset S$,
	\end{itemize}
where $I,I' \in \Jc(P) \setminus \{\emptyset \}$ with $I \neq I'$ and $S,S' \in S(G) \setminus \{\emptyset\}$ with $S \neq S'$,
	and set
	$$\Mc_{\Cc\Qc}=\Mc_{\Cc_P} \cup \Mc_{\Qc_G} \cup \{x_{\max(I)}y_S : I \in \Jc(P),S \in S(G), \text{max}(I) \cap S \neq \emptyset\}.$$
	%Let $K[[-B, A]^{*}] \subset K[ t_{1}^{\pm 1}, \ldots, t_{d+1}^{\pm 1}, s]$
	%be the toric ring of $[-B, A]^{*}$
	Let $\Gc$ be a finite set of binomials belonging to $I_{\Gamma(\Cc_P,\Qc_G)}$ with ${\mathcal M}_{\Cc\Qc} = \{ {\rm in}_{<}(g) : g \in {\mathcal G}\}$.
	By the same way of the proof of Proposition \ref{gamma}, we can prove that $\Gc$ is a Gr\"obner base of ${\Omega(\Cc_P,\Qc_G)}$ with respect to $<$.
	
	Now, use the same notation as in the proof of Proposition \ref{gamma}.
	Set
	$$R_{\Oc\Qc}=\dfrac{K[\Oc\Qc]}{(M_{\Oc\Qc})}, R_{\Cc\Qc}=\dfrac{K[\Cc\Qc]}{(M_{\Cc\Qc})}.$$
	Then by Lemma \ref{lem:hilbert}, the Hilbert function of $K[\Gamma(\Oc_P,\Qc_G)]$ equals that of $R_{\Oc\Qc}$, and the Hilbert function of $K[\Gamma(\Cc_P,\Qc_G)]$ equals that of $R_{\Cc\Qc}$.
	Moreover, it is easy to see that the ring homomorphism $\phi : R_{\Oc\Qc} \to R_{\Cc\Qc}$ by setting
	$\phi(x_I)=x_{\max(I)}, \phi(y_S)=y_S$ and $\phi(z) =z$ is an isomorphism, where $I \in \Jc(P) \setminus \{\emptyset\}$ and $S \in S(G) \setminus \{\emptyset\}$.
	Hence since $\Gamma(\Oc_P,\Qc_G)$ and $\Gamma(\Cc_P,\Qc_G)$ have the integer decomposition property, we obtain
	$$\delta(\Gamma(\Oc_P,\Qc_G),\lambda)=\delta(\Gamma(\Cc_P,\Qc_G),\lambda).$$
	Similarly, one has
	$$\delta(\Omega(\Oc_P,\Qc_G),\lambda)=\delta(\Omega(\Cc_P,\Qc_G),\lambda).$$
	Moreover, since the chain polytope $\Cc_P$ is a stable set polytope of a perfect graph, by Lemma \ref{delta}, it follows that 
	$$\delta(\Omega(\Oc_P,\Qc_G),\lambda)=(1+\lambda) \cdot \delta(\Gamma(\Oc_P,\Qc_G),\lambda),$$
	as desired.
\end{proof}

\end{document}